\documentclass[11pt,a4paper]{article}%
\usepackage{amssymb}
\usepackage{amsfonts}
\usepackage{amsmath}
\usepackage{verbatim}
\usepackage{graphicx}%
\newtheorem{theorem}{Theorem}

\newtheorem{corollary}[theorem]{Corollary}

\newtheorem{remark}[theorem]{Remark}

\newenvironment{proof}[1][Proof]{\noindent\textbf{#1.} }{\ \rule{0.5em}{0.5em}}
\begin{document}

\title{Infinite family of equi-isoclinic planes in Euclidean odd dimensional spaces and of complex symmetric conference matrices of odd orders }
\author{Boumediene Et-Taoui\\Universit\'{e} de Haute Alsace -- LMIA\\4 rue des fr\`{e}res Lumi\`{e}re 68093 Mulhouse Cedex}
\maketitle

\begin{abstract}
A $n$-set of equi-isoclinic planes in $\mathbb{R}^{r}$ is a set of $n$ planes spanning $\mathbb{R}^{r}$ each pair of which has the same non-zero angle $\arccos\sqrt\lambda$.
We prove that for any odd integer $k\geq 3$ such that $2k=p^{\alpha}+1$, $p$ odd prime, $\alpha$ non-negative integer the maximum number of
equi-isoclinic planes with angle $\arccos\sqrt\frac{1}{2k-2}$ in $\mathbb{R}^{2k-1}$ is equal to $2k-1$.
The solution of this geometric problem is obtained by the construction of complex symmetric conference matrices of order $2k-1$.


\end{abstract}

B.Ettaoui@uha.fr \bigskip

\section{Introduction}

\subsection{Equi-isoclinic planes and complex conference matrices}
This paper deals with the so-called notion of \emph{isoclinic} planes. In an
Euclidean space of dimension $r$, $r\geq4$, two planes have $2$ angles. They are the stationary
values of the angle between the lines $l$ and $m$, if $m$ runs through one
plane and $l$ runs through the other plane. If the two angles are equal (\textit{%
i.e.} the angle between $l$ and $m$ is constant), then the two planes
are said to be \emph{isoclinic}. If all the planes of a given $n$-tuple are
pairwise isoclinic with the same angle $\phi$ then the planes are
said to be \emph{equi-isoclinic}, and the number $\lambda=\cos^{2} \phi $ is called the
\emph{parameter} of the tuple.
In \cite{LS1} the authors pose the problem of finding the maximum number $%
v(2,r)$ of \emph{equi-isoclinic} planes that can be imbedded in $\mathbb{R}^{r}$, and more precisely, the maximal number $v_{\lambda}\left(
2,r\right) $ of \emph{equi-isoclinic} planes in $\mathbb{\mathbb{R}}^{r}$ with
the \emph{parameter} $\lambda$. They proved that
\begin{equation}
(2-r\lambda)v{_{\lambda}}(2,r)\leq r(1-\lambda)\text{.}  \label{1}
\end{equation}
Among other things they showed that $v(2,4)=v_{\frac{1}{3}}(2,4)=4$.
In \cite{Id0}, the author gives the values of $v_{\lambda}\left( 2,2r\right) $ for
some infinite family of ordered pairs $\left( \lambda,r\right) $; however, for odd
integers $r$, no value of $v$ was known yet. The first example in odd
dimensional space is given in \cite{Id}, where it is shown that $v\left(
2,5\right) =v_{\frac{1}{4}}\left(2,5\right)= 5$. It is established in this paper that for any odd integer $k\geq 3$ such that $2k=p^{\alpha}+1$, $p$ odd prime,
$\alpha$ non-negative integer, $v_{\frac{1}{2k-2}}\left( 2,2k-1\right) =2k-1$. The case $k=3$ , handles the present author to derive $v(2,5)$ \cite{Id}.
It is seen that our geometric problem amounts to find some complex symmetric square matrices of odd orders which are called complex conference matrices.
A complex $n\times n$ \textit{conference matrix} $C$ is a matrix
with $c_{ii}=0$ and $\left\vert c_{ij}\right\vert =1$, $i\neq j$ that
satisfies
\begin{equation}
CC^{\ast}=(n-1)I_{n} \label{2}.
\end{equation}
Real \emph{conference} matrices have been heavily studied in the literature in
connection with combinatorial designs in geometry, engineering, statistics,
and algebra. In particular questions in the theory of polytopes, posed by Coxeter (\cite {Co}), led Paley (\cite {P}) to the construction of real Hadamard matrices in which he used \emph{conference} matrices.
The following necessary conditions are known : $n\equiv2$ $(\operatorname{mod}%
4)$ and $n-1=a^{2}+b^{2}$, $a$ and $b$ integers for symmetric matrices
(\cite{B}, \cite{LS}, \cite{P}, \cite{R}), and $n=2$ or $n\equiv0$ $(\operatorname{mod}%
4)$ for skew symmetric matrices \cite{W}. For the values of $n$ for which there exist a real \emph{conference} matrix of order $n$ we refer to
(\cite{GS}, \cite{M}, \cite{P}, \cite{SW}, \cite{W}). In \cite{DGS} the authors showed that essentially there are no other
real \emph{conference} matrices. Precisely they proved that any \emph{real conference} matrix
of order $n>2$ is equivalent, under multiplication of rows and columns by
$-1$, to a \emph{conference} symmetric or to a skew symmetric matrix according as $n$
satisfies $n\equiv2$ $(\operatorname{mod}4)$ or $n\equiv0$
$(\operatorname{mod}4)$. In addition we observe that $n$ must be even.

This is not the case for complex \emph{conference} matrices (\ref{2}). They were used in \cite{D} to provide
parametrization of complex Sylvester inverse orthogonal matrices. It is quoted
in \cite{D}, and easy to show, that there is no complex conference matrix of
order $3$; however we can find such a matrix of order $5$ \cite{D}. Some complex conference matrices of even orders can be easily constructed.
For example starting from a real symmetric (or skew symmetric) \emph{conference} matrix $C$ multiplication by the complex number $i$ yields a complex symmetric (or Hermitian) \emph{conference} matrix
of even order. However no method to construct complex conference matrices of odd orders is known yet.

In this article, we construct an infinite family of complex symmetric \emph{conference}
matrices of odd orders from which we deduce an infinite family of \emph{equi-isoclinic} planes in Euclidean odd dimensional spaces. The Paley real symmetric \emph{conference} matrices are  obtained from the  Paley symmetric matrices $P$ of order $2k-1=p^{\alpha}$, $p$ odd prime, $\alpha$ non-negative integer, such that
\begin{center}
$P^{2}=(2k-1)I-J$, and
\end{center}
\begin{center}
$PJ=JP=0$.
\end{center}
It appears that for each order for which there exists a symmetric Paley matrix there exists a complex symmetric \emph{conference} matrix with that order.
Although the two constructions are both obtained by use of the  Legendre symbol of the Galois field $GF(p^{\alpha})$ they differ by use of the Jacobsthal's theorem (\cite{W})
in the Paley case and of an analogous theorem given below (\ref{5}) in our case. Complex \emph{conference} matrices are also important in complex Hadamard matrix theory because if
$C_{n}$ is a complex conference matrix of order $n$ then by construction the matrix
\begin{equation}
H_{2n}=\left(
\begin{array}
[c]{cc}%
C_{n}+I_{n} & C_{n}^{\ast}-I_{n}\\
C_{n}-I_{n} & -C_{n}^{\ast}-I_{n}%
\end{array}
\right) \label{3},
\end{equation}
is a complex Hadamard matrix of order $2n$.
A matrix of order $n$ with unimodular complex entries and satisfying $HH^{\ast}%
=nI_{n}$ is called \textit{complex Hadamard}.

The maximum $5$-tuple in $\mathbb{R}^{5}$ shares with the five-order complex symmetric \emph{conference} matrix the property that the square of its Seidel's associated matrix is a
multiple of the identity matrix.  This property is here generalized as follows: for any integer $k$ such that $2k=p^{\alpha}+1\equiv2$
$(\operatorname{mod}4)$, $p$ odd prime, $\alpha$ non-negative integer on the one hand we construct new, previously
unknown infinite family of complex symmetric conference matrices of order $2k-1$ and on the other hand we prove that $v_{\frac{1}{2k-2}}\left( 2,2k-1\right) =2k-1$.
It turns out that each complex symmetric conference matrix of order $2k-1$ depends on a unit complex number $\omega_{0}$ such that $\Re(\omega_{0}^{2})=\frac{2-k}{k-1}$.

\subsection{Seidel's matrices}

The main tool for the determination of $v\left( 2,r\right) $ is the notion
of Seidel's matrices. If the planes $\Gamma {_{1}}$,$\Gamma {_{2}}$,...,$%
\Gamma {_{n}}$ are each provided with an orthonormal basis, then the matrix $%
A{_{ij}}$ built up by the inner products of vectors in the basis of $\Gamma {%
_{i}}$ with vectors in the basis of $\Gamma {_{j}}$ satisfies $A{_{ij}}A{%
_{ji}}=\lambda I_{2}$, $i\neq j=1$,...,$n$ \cite{LS1}. Hence the block matrix%
\begin{equation*}
A=\left(
\begin{array}{cccc}
I_{2} & A_{12} & \cdots  & A_{1n} \\
A_{21} & I_{2} & \ddots  & \vdots  \\
\vdots  & \ddots  & \ddots  & A_{n-1~n} \\
A_{n1} & \cdots  & A_{n~n-1} & I_{2}%
\end{array}%
\right) \text{,}
\end{equation*} %
of order $2n$, is symmetric positive semi-definite and of rank $r$. Thus,
in order to investigate $n$-tuples of \textit{equi-isoclinic} planes
with parameter $\lambda$ in $\mathbb{R}^{r}$ we ask for symmetric block matrices,
that are positive semi-definite with rank $r$, i.e. that have the smallest
eigenvalue $0$ with multiplicity $2n-r$. In other words, we ask for
symmetric matrices $S=\frac{1}{\sqrt\lambda}(A-I_{2n})$ whose smallest eigenvalue has
multiplicity $2n-r$.

The matrix $S$ is then partitioned into square blocks $(S_{ij})$ of order $2$
with $S_{ii}=0\in M_{2}\left( \mathbb{R}\right) $ for all $i=1$,~...,$~n$,
and $S_{ij}\in O(2)$ for all $i\neq j$.

Such matrices are called Seidel's matrices in memory of J.J. Seidel, who,
together with P.W.H. Lemmens, has first introduced them.

Conversely, any Seidel's matrix with smallest eigenvalue $\mu _{0}$ whose
multiplicity is $2n-r$ can always be obtained from some $n$-tuple of \textit{%
equi-isoclinic} planes with the parameter $\lambda=\frac{1}{\mu _{0}^{2}}$
imbedded in $\mathbb{R}^{r}$ \cite{LS1}.

Furthermore an $n$-tuple of \emph{equi-isoclinic} planes is not characterized by a
single matrix. We will say that two Seidel's matrices $S$ and $S^{\prime }$
are equivalent if and only if they are associated to the same $n$-tuple. It is
proved that two Seidel's matrices $S=\left( S_{ij}\right) _{ij}$ and $%
S^{\prime }=\left( S_{ij}^{\prime }\right) _{ij}$ are equivalent if and only
if there exit $n$ matrices $P_{1},\ldots ,P_{n}\in O\left( 2\right) $ such
that $P_{i}^{T}S_{ij}P_{j}=S_{ij}^{\prime }$ \cite{Id}, \cite{EF}.

By multiplication of each block column $j=2,...,n$ by $S_{j1}$ and
corresponding row by $S_{1j}$ each Seidel's matrix is equivalent to one
which has\textbf{\ }(apart from $S_{11}=0$) $I_{2}$'s on the first row and
column. This matrix is said to be \textit{normal}. It has been shown in \cite%
{EF} that two \emph{normal} Seidel's matrices $S$ and $S^{\prime }$ are equivalent
if and only if there is a matrix $P\in O\left( 2\right) $ such that each
block of $S$ is conjugated to the corresponding block of $S^{^{\prime }}$ by
$P$.

The relationship between Seidel's matrices associated  to \emph{equi-isoclinic} planes in Euclidean odd dimensional spaces and the odd orders complex
symmetric \emph{conference} matrices is the following. Let $C$ be any complex symmetric \emph{conference} matrix of odd order $n$,
then replace each zero of the diagonal by the zero matrix of order two and each all other non diagonal entrie which is a unit complex number $e^{i\theta_{\alpha\beta}}$ by the two-order matrix $\left(
                                                             \begin{array}{cc}
                                                               \cos(\theta_{\alpha\beta}) & \sin(\theta_{\alpha\beta}) \\
                                                               \sin(\theta_{\alpha\beta}) & -\cos(\theta_{\alpha\beta}) \\
                                                             \end{array}
                                                           \right),$
the matrix of a plane symmetry denoted by $s_{\theta_{\alpha\beta}}$ . This relationship is performed by showing that the obtained matrix $S$ satisfies $S^{2}=(n-1)I_{2n}$.

In section $1$ we construct our infinite family of complex symmetric \emph{conference} matrices of odd orders. Section $2$ is devoted to the construction of Seidel's marices associated to tuples of \emph{equi-isoclinic} planes which generate Euclidean spaces of odd dimension.

As for the notations, $I$ and $J$ denote, the unit and  all-one matrices. The matrix $A^{*}$ denotes the adjoint matrix of $A$ and, unless otherwise specified, all matrices are of order $q=2k-1$.
\section{Constructing odd orders complex symmetric conference matrices from finite fields}

\subsection{Finite fields and the Legendre Symbol}
Let $GF(q)$ be the Galois field of order $q=p^{\alpha}$, $p^{\alpha}\equiv1$ $(\operatorname{mod}4)$, $p$ odd prime, $\alpha$ non-negative integer. Let
$\varkappa$ denote the Legendre symbol, defined by $\varkappa(0)=0$,
$\varkappa(x)=1$ or $-1$ according as $x$ is or not a square in $GF(p^{\alpha
})$. It is well known that $\varkappa(-1)=1$,
$\varkappa(xy)=\varkappa(x)\varkappa(y)$,
$\varkappa(x^{-1})=(\varkappa(x))^{-1}$, $\varkappa(-x)=\varkappa(x)$ and $\sum\limits_{x \in GF(q)}^{}\varkappa(x)=0$.
Let $q=2k-1$, $\omega$ be any complex number of modulus one, $(a_{\alpha})$, $\alpha=1,...,q$ be the elements of $GF(q)$, and define the square matrix $C(\omega)$ of order $q$ by
\begin{center}
$c_{\alpha\alpha}=0$, $\alpha=1,...,q$, and
\end{center}
\begin{equation}
c_{\alpha\beta}=\omega^{\varkappa(a_{\alpha}-a_{\beta})},\\ \alpha\neq\beta,\\ \alpha,\beta=1,...,q \label{4} .
\end{equation}
\subsection{An amazing formula}
Here is an analogous theorem to Jacobsthal's theorem which was used by Paley in his construction of real symmetric conference matrices.
\begin{theorem}
For any $b\in GF(q)^{*}$ we have
\begin{equation}
\sum\limits_{a \in GF(q)^{*}\backslash \{-b\}
}^{}\omega^{\varkappa(a)-\varkappa(a+b)}=k-2+(k-1)\Re(\omega^{2}) \label{5}.
\end{equation}
\end{theorem}
\begin{proof}
Define $z=1+\frac{b}{a}$. As $a$ ranges over all non-zero elements of GF(q) except $-b$, $z$ ranges over all non-zero elements of GF(q) except unity.
However
\begin{center}
$\frac{1-\varkappa(z)}{\varkappa(z-1)}=\frac{\varkappa(a)-\varkappa(a+b)}{\varkappa(b)}$.
\end{center}
Thus
\begin{center}
$\sum\limits_{a \in GF(q)^{*}\backslash \{-b\}}^{}\omega^{\varkappa(a)-\varkappa(a+b)}=\sum\limits_{z \in GF(q)^{*}\backslash \{1\}}^{}\omega^{\varkappa(b)\frac{1-\varkappa(z)}{\varkappa(z-1)}}$.
\end{center}
We claim that this last sum is independent of $b$. Indeed
\begin{center}
$\sum\limits_{z \in GF(q)^{*}\backslash \{1\}}^{}\omega^{-\frac{1-\varkappa(z)}{\varkappa(z-1)}}=\sum\limits_{z \in GF(q)^{*}\backslash \{1\}}^{}\omega^{\frac{1-\varkappa(z)}{\varkappa(z-1)}}$,
\end{center}
put in the first sum $z'=\frac{1}{z}$.

Recall that there are $k-1$ non-zero squares and $k-1$ non-zero non-squares. Clearly our sum is then of the form
\begin{center}
$r.1+s\omega^{2}+t\omega^{-2}$.
\end{center}
Whence $r=k-2$. Now $\omega^{2}$ and $\omega^{-2}$ appear in pairs because if a non-square $z$ yields $\omega^{2}$ (respectively $\omega^{-2}$) then
its inverse which is also a non-square yields $\omega^{-2}$ (respectively $\omega^{2}$). According $s=t=\frac{k-1}{2}$. This proves the assertion.
\end{proof}
\subsection{Complex symmetric conference matrices of odd orders}
\begin{theorem}
The complex symmetric matrix $C(\omega)$ of order $2k-1$ satisfies
\begin{center}
$CC^{*}=(2k-2-c)I+cJ$, with
\end{center}
\begin{center}
$c=k-2+(k-1)\Re(\omega^{2})$.
\end{center}
\end{theorem}
\begin{proof}
Let us denote $(d_{\alpha\beta})$ for $\alpha,\beta=1,...,q$ the entries of the product $CC^{*}$. Then for $\alpha=1,...,q$, $d_{\alpha\alpha}=q-1$  and for $\alpha,\beta=1,...,q$,
$\alpha\neq\beta$,
\begin{center}
$d_{\alpha\beta}=\sum\limits_{a_{\gamma} \in GF(q)\backslash \{a_{\alpha},a_{\beta}\}^{}}\omega^{\varkappa(a_{\alpha}-a_{\gamma})-\varkappa(a_{\gamma}-a_{\beta})}$.
\end{center}
Clearly
\begin{center}
$d_{\alpha\beta}=\sum\limits_{a \in GF(q)^{*} \backslash \{-b\}}^{}\omega^{\varkappa(a)-\varkappa(a+b)}$,
\end{center}
with $b\in GF(q)^{*}$. According to (\ref{5})
\begin{center}
$d_{\alpha\beta}=k-2+(k-1)\Re(\omega^{2})$.
\end{center}
This completes the proof of the theorem.
\end{proof}

As a consequence we obtain the following the result.
\begin{corollary}
The complex symmetric matrix $C(\omega_{0})$ of order $2k-1$ and with $\Re(\omega_{0}^{2})=\frac{2-k}{k-1}$ satisfies
\begin{equation}
CC^{*}=(2k-2)I \label{6}.
\end{equation}
\end{corollary}
Here are the $5$-order and the $9$-order complex symmetric conference matrices :
\begin{center}
$\left(
   \begin{array}{ccccc}
     0 & j & j^{2} & j^{2} & j \\
     j & 0 & j & j^{2} & j^{2} \\
     j^{2} & j & 0 & j & j^{2} \\
     j^{2} & j^{2} & j & 0 & j \\
     j & j^{2} & j^{2} & j & 0 \\
   \end{array}
 \right)$,
\end{center}
with $j=e^{\frac{i2\pi}{3}}$, and
\begin{center}
$\left(
  \begin{array}{ccccccccc}
    0 & \omega_{0} & \omega_{0} & \omega_{0} & \omega_{0} & \omega_{0}^{-1} & \omega_{0}^{-1} & \omega_{0}^{-1} & \omega_{0}^{-1} \\
    \omega_{0} & 0 & \omega_{0}^{-1} & \omega_{0}^{-1} & \omega_{0} & \omega_{0} & \omega_{0} & \omega_{0}^{-1} & \omega_{0}^{-1} \\
    \omega_{0} & \omega_{0}^{-1} & 0 & \omega_{0} & \omega_{0}^{-1} & \omega_{0} & \omega_{0}^{-1} & \omega_{0} & \omega_{0}^{-1} \\
    \omega_{0} & \omega_{0}^{-1} & \omega_{0} & 0 & \omega_{0}^{-1} & \omega_{0}^{-1} & \omega_{0} & \omega_{0}^{-1} & \omega_{0} \\
    \omega_{0} & \omega_{0} & \omega_{0}^{-1} & \omega_{0}^{-1} & 0 & \omega_{0}^{-1} & \omega_{0}^{-1} & \omega_{0} & \omega_{0} \\
    \omega_{0}^{-1} & \omega_{0} & \omega_{0} & \omega_{0}^{-1} & \omega_{0}^{-1} & 0 & \omega_{0} & \omega_{0} & \omega_{0}^{-1} \\
    \omega_{0}^{-1} & \omega_{0} & \omega_{0}^{-1} & \omega_{0} & \omega_{0}^{-1} & \omega_{0} & 0 & \omega_{0}^{-1} & \omega_{0} \\
    \omega_{0}^{-1} & \omega_{0}^{-1} & \omega_{0} & \omega_{0}^{-1} & \omega_{0} & \omega_{0} & \omega_{0}^{-1} & 0 & \omega_{0} \\
    \omega_{0}^{-1} & \omega_{0}^{-1} & \omega_{0}^{-1} & \omega_{0} & \omega_{0} & \omega_{0}^{-1} & \omega_{0} & \omega_{0} & 0 \\
  \end{array}
\right)$,
\end{center}

where $\omega_{0}$ is a unit complex number such that $\Re(\omega_{0}^{2})=-\frac{3}{4}$.

The following operations on the set of complex symmetric conference matrices of order $q$:
\begin{enumerate}
\item multiplication by any unit complex number of any row and the corresponding column,
\item interchange of rows and, simultaneously, of the corresponding columns,
\end{enumerate}
generate a relation , called \emph{equivalence}. The second operation alone also generates a relation called \emph{permutation equivalence}.
Concerning this equivalence relation we have the following result.
\begin{theorem}
For any $q=2k-1=p^{\alpha}$, $p$ odd prime, $k\geq3$ there exists a class of equivalent complex symmetric conference matrices of order $q$.
\end{theorem}
\begin{proof}
For any $q$ There are four possible unit complex numbers $\omega$ for which the matrices $C(\omega)$ satisfy (\ref{6}). They are $\omega_{0}$, $\omega_{0}^{-1}$, $-\omega_{0}$ and
$-\omega_{0}^{-1}$. Consider the matrix $C(\omega_{0}^{-1})$ whose non diagonal entries are
\begin{center}
$c_{\alpha\beta}=\omega_{0}^{-\varkappa(a_{\alpha}-a_{\beta})}$, $\alpha\neq\beta$, $\alpha,\beta=1,...,q$.
\end{center}
Then for some non-square $a_{\gamma}$ all $a_{\alpha}a_{\gamma}$ are distinct and $C(\omega_{0}^{-1})=(\omega_{0}^{\varkappa(a_{\alpha}a_{\gamma}-a_{\beta}a_{\gamma})})$
is permutation equivalent to $C(\omega_{0})$. Consider now the matrix $C(-\omega_{0})$. Using our definition (\ref{4}) it is easily seen that the matrix $C(-\omega_{0})$ is
equal to $-C(\omega_{0})$ which is clearly equivalent to $C(\omega_{0})$, by multiplication
by the complex number $i$ of each row and of the corresponding column. This completes the proof of the theorem.
\end{proof}
\section{Seidel's matrices with plane symmetries }

\subsection{Seidel's matrices and a second amazing formula}
Let us denote $\omega_{0}=e^{i\theta}$ with $\cos(2\theta)=\frac{2-k}{k-1}$ and define the block matrix $S$ of order $2(2k-1)$ as follows :
\begin{center}
$S_{\alpha\alpha}=\left(
                     \begin{array}{cc}
                       0 & 0 \\
                       0 & 0 \\
                     \end{array}
                   \right)$,
\end{center}
$\alpha=1,...,q$, and
\begin{center}
$S_{\alpha\beta}=\left(
                   \begin{array}{cc}
                     \cos(\theta\varkappa(a_{\alpha}-a_{\beta})) & \sin(\theta\varkappa(a_{\alpha}-a_{\beta})) \\
                     \sin(\theta\varkappa(a_{\alpha}-a_{\beta})) & -\cos(\theta\varkappa(a_{\alpha}-a_{\beta})) \\
                   \end{array}
                 \right)$,
\end{center}
$\alpha\neq\beta$, $\alpha,\beta=1,...,q$.
As in case of complex symmetric conference matrices of odd orders we need the following formula.
\begin{theorem}
Let $\theta$ be any real number, $r_{\eta}$ the plane rotation with angle $\eta$ and $b \in GF(q)^{*}$. Then
\begin{equation}
\sum\limits_{a \in GF(q)^{*}\backslash \{-b\}}^{}r_{\theta(\varkappa(a)-\varkappa(a+b))}=(k-2+(k-1)\cos(2\theta))I_{2} \label{7}.
\end{equation}
\end{theorem}
\begin{proof}
The proof is the same word by word as the proof of Theorem $1$. For the last statement we just need the equality $r_{2\theta}+r_{-2\theta}=2\cos(2\theta)I_{2}$.
\end{proof}
\begin{theorem}
The matrix $S$ of order $4k-2$ is symmetric and satisfies
\begin{equation}
S^{2}=(2k-2)I_{4k-2} \label{8}.
\end{equation}
\end{theorem}
\begin{proof}
We denote by $T$ the block matrix $S^{2}$. Then clearly for $\alpha=1,...,2k-1$, $T\alpha\alpha=(2k-2)I_{2}$  and for $\alpha,\beta=1,...,2k-1$, $\alpha\neq\beta$,
\begin{center}
$T_{\alpha\beta}=\sum\limits_{a_{\gamma} \in GF(q)\backslash \{a_{\alpha},a_{\beta}\}}^{}s_{\theta\varkappa(a_{\alpha}-a_{\gamma})}s_{\theta\varkappa(a_{\gamma}-a_{\beta})}$.
\end{center}
However
\begin{center}
$s_{\theta\varkappa(a_{\alpha}-a_{\gamma})}s_{\theta\varkappa(a_{\gamma}-a_{\beta})}=r_{\theta(\varkappa(a_{\alpha}-a_{\gamma})-\varkappa(a_{\gamma}-a_{\beta}))}$.
\end{center}
As in Theorem $2$ we have
\begin{center}
$T_{\alpha\beta}=\sum\limits_{z \in GF(q)^{*}\backslash \{1\}}^{}r_{\theta\varkappa(b)\frac{1-\varkappa(z)}{\varkappa(z-1)}}$.
\end{center}
Thus
\begin{center}
$T_{\alpha\beta}=\sum\limits_{z \in GF(q)^{*}\backslash \{1\}}^{}r_{\theta\frac{1-\varkappa(z)}{\varkappa(z-1)}}$.
\end{center}
The same arguments as in Theorem $1$ and (\ref{7}) lead to
\begin{center}
$T_{\alpha\beta}=(k-2+(k-1)\cos(2\theta))I_{2}$.
\end{center}
Since $\cos(2\theta)=\frac{2-k}{k-1}$ we obtain $T_{\alpha\beta}=\left(
                     \begin{array}{cc}
                       0 & 0 \\
                       0 & 0 \\
                     \end{array}
                   \right)$.
\end{proof}
\begin{corollary}
$S$ has two eigenvalues with equal multiplicities.
\end{corollary}
The proof is a consequence from the equality (\ref{8}) satisfied by $S$ and by the fact that it has zero trace.
\subsection{Equi-isoclinic planes in Euclidean odd dimensional spaces}
\begin{corollary}
For any integer $k\geq3$ such that $2k=p^{\alpha}+1\equiv2$
$(\operatorname{mod}4)$, $p$ odd prime, $\alpha$ non-negative integer  $v_{\frac{1}{2k-2}}\left( 2,2k-1\right) =2k-1$.
\end{corollary}
\begin{proof}
The matrix $S$ is a Seidel's matrix of order $2(2k-1)$, its smallest eigenvalue $-\sqrt{2k-2}$ has multiplicity $2k-1$ and then leads to a $(2k-1)$-tuple of equi-isoclinic planes with parameter $\frac{1}{2k-2}$ and which span a $(2k-1)$-dimensional Euclidean space. It follows that
\begin{center}
$v_{\frac{1}{2k-2}}\left( 2,2k-1\right) \geq 2k-1$.
\end{center}
Now using Lemmens and Seidel's inequality (\ref{1}) we arrive to
\begin{center}
$v_{\frac{1}{2k-2}}\left( 2,2k-1\right) \leq 2k-1$,
\end{center}
and the corollary follows.
\end{proof}
\begin{corollary}
$v_{\frac{1}{4}}\left(2,5\right)= 5$, $v_{\frac{1}{8}}\left(2,9\right)= 9$, $v_{\frac{1}{12}}\left(2,13\right)= 13$....
\end{corollary}
In order to give more examples, we look for all odd $k\geq3$ and $k\leq 51$  such that $2k=p^{\alpha}+1$, $p$ odd prime, $\alpha$ non-negative integer, we obtain the following corollary.
\begin{corollary}
If $k$ is odd, $3\leq k\leq51$, we may construct the $(2k-1)$-order complex symmetric conference matrix and thus the maximal $(2k-1)$-tuple of equi-isoclinic planes with parameter $\frac{1}{2k-2}$ in
$\mathbb{R}^{2k-1}$, except possibly in the cases $k=11,17,23,29,33,35,39,43,47$.
\end{corollary}
\begin{proof}
Remark that Paley real symmetric matrices and our complex symmetric conference matrices exist for the same orders. The real conference matrix of order $46$ is not
from the Paley type that was found in \cite{M}. For the cases $k=11,17,29,35,39,47$, the integers $2k$ do not meat the necessary conditions for the existence of real symmetric
conference matrices but it does not mean that there is no complex symmetric conference matrices with orders $2k-1$.
The cases $66,86$ meat the necessary conditions but it is not yet known if we may construct real symmetric conference matrices with these orders which is also true for
complex symmetric conference matrices with orders $65$ and $85$. From this the corollary follows.
\end{proof}

The relationship described above induces a relation between the equivalence relations defined on complex symmetric conference matrices of order $q$ and on corresponding Seidel's matrices of order $2q$ as
follows: the operation which consists in multiplying
the row $\alpha$ and the corresponding column of a complex symmetric conference matrix by any unit complex number $e^{i\eta}$ reads on the corresponding Seidel' matrix as follows: multiply the
block row $\alpha$ by the plane rotation $r_{\frac{\eta}{2}}$ and the corresponding block column by $r_{-\frac{\eta}{2}}$. Indeed we have the equality:
for any real $\eta$,
\begin{equation}
r_{\frac{\eta}{2}}s_{(\pm\theta)}r_{-\frac{\eta}{2}}=s_{(\eta\pm\theta)},
\end{equation}
where $\theta$ is a real such that $\cos(2\theta)=\frac{2-k}{k-1}$. If one multiplies a complex symmetric conference matrix $C(\omega_{0})$ of order $q$ on the left by the unitary diagonal matrix
$diag(e^{i\alpha_{1}},...,e^{i\alpha_{q}})$ and on the right by $diag(e^{i\alpha_{1}},...,e^{i\alpha_{q}})$ we obtain a complex symmetric conference matrix of order $q$ to which corresponds a Seidel's matrix of order $2q$.
This last matrix is obtained by multiplication of the Seidel's matrix $S(\omega_{0})$ corresponding to $C(\omega_{0})$, on the left by $diag(r_{\alpha_{1}},..., r_{\alpha_{q}})$ and on the right
by $diag(r_{-\alpha_{1}},..., r_{-\alpha_{q}})$.
From this follows that if two complex symmetric conference matrices of order $q$ are equivalent then the corresponding Seidel's matrices are equivalent.
According, for any integer $k\geq3$ such that $2k=p^{\alpha}+1\equiv2(\operatorname{mod}4)$, $p$ odd prime, $\alpha$ non-negative integer there is a class of
congruent $(2k-1)$-tuples of equi-isoclinic planes with parameter $\frac{1}{2k-2}$ and which span a $(2k-1)$-dimensional Euclidean space. Indeed the congruence class of a $n$-tuple of
equi-isoclinic planes in the Euclidean space is determined by the equivalence class of its associated Seidel's matrix (\cite{Id}, \cite{EF}).
\begin{remark}
The associated Seidel's matrix of the $5$-tuple of the Euclidean space $\mathbb{R}^{5}$ found in \cite{Id} is in normal form. However our $10$-order Seidel's matrix obtained by Theorem $6$ is
not in normal form but can easily be brought into the first one by use of the equivalence relation defined in the above introduction and eventually the operation which consists in
permuting block rows and simultaneously the corresponding block columns.
Complex Hermitian conference matrices of even orders were used in \cite{Id0} to show that $v_{\frac{1}{2k-1}}\left( 2,2k\right)=2k$, for any $k$ for which there exists a
real skew symmetric conference matrix $C$ of order $2k$ (\cite{P}, \cite{W}). That is the matrix $iC$ is a complex Hermitian conference matrix.
Among other things the present author showed in \cite{Id1} that $v\left(2,6\right)=v_{\frac{1}{4}}\left(2,6\right)= 9$ and that if $\lambda>\frac{1}{4}$, the maximum number of equi-isoclinic planes in $\mathbb{R}^{2r}$ with parameter $\lambda$ is equal to that one of equiangular lines in $\mathbb{C}^{r}$ with angle $\arccos\sqrt\lambda$.

These are all the known results on $v(2,r)$. Of course our results do not lead to the exact values of $v(2,2k-1)$ for the concerned values of $k$ given above except the case $k=3$ but we hope that they can be
considered as a progress for the future since the odd dimension stays so mysterious in our research as in other subjects.

Our odd-orders complex conference matrices can be used efficiently to obtain
previously unknown complex Hadamard matrices by doubling their seize (\ref{3}) and they may be important in quantum information theory. These
matrices can be also used as starting points to construct Sylvester inverse
orthogonal matrices by the method developed in \cite{D}.
\end{remark}

\end{document}